\documentclass[10pt]{amsart}

\usepackage{amsmath, amsfonts, amsthm, amssymb, graphicx, fullpage, enumerate, float, caption, subcaption, tikz, hyperref, scalerel, boldline, makecell, longtable, array, tabularx}
\usetikzlibrary{graphs}
\usepackage{algorithm}
\usepackage{algpseudocode}   
    
\newtheorem{theorem}{Theorem}[section]        
\newtheorem{lemma}[theorem]{Lemma}

\newtheorem{conjecture}[theorem]{Conjecture}

\theoremstyle{remark}

\theoremstyle{definition}  
\newtheorem{definition}[theorem]{Definition}   
      
\def \epsilon{\varepsilon}
\def \l{\ell}

\def\N{\mathbb{N}}

\def\Z{\mathbb{Z}}

\def\P{\mathcal{P}}

\def\bar#1{\overline{#1}} 

\begin{document}
\title{Notes and Computations On Forbidden Differences}  
\author{Christian Dean, \quad Haley Havard, \quad Elizabeth Hawkins, \\ Patch Heard, \quad Andrew Lott, \quad Alex Rice}
 
\begin{abstract} We explore from several perspectives the following question: given $X\subseteq \Z$ and $N\in \N$, what is the maximum size $D(X,N)$ of $A\subseteq \{1,2,\dots,N\}$ before $A$ is forced to contain two distinct elements that differ by an element of $X$? The set of forbidden differences, $X$, is called \textit{intersective} if $D(X,N)=o(N)$, with the most well-studied examples being $X=S=\{n^2: n\in \N\}$ and $X=\P-1=\{p-1: p\text{ prime}\}$. In addition to some new results, including exact formulas and estimates for $D(X,N)$ in some non-intersective cases like $X=\P$ and $X=S+k$, $k\in \N$, we also provide a comprehensive survey of known bounds and extensive computational data. In particular, we utilize an existing algorithm for finding maximum cliques in graphs to determine $D(S,N)$ for $N\leq 300$ and $D(\P-1,N)$ for $N\leq 500$. None of these exact values appear previously in the literature.
  
\end{abstract}

\address{Department of Mathematics, Millsaps College, Jackson, MS 39210}   
\email{deanmchris@gmail.com} 
\email{hawkie@millsaps.edu}   
\email{patchh@icloud.com} 
\email{riceaj@millsaps.edu}

\address{Department of Mathematics, University of Georgia, Athens, GA 30602}  
\email{heh25323@uga.edu} 
\email{andrew.lott@uga.edu}

\maketitle   
\setlength{\parskip}{5pt}   
 
\section{Introduction} \label{intro}

Dating at least to the 1970s, a popular family of questions and results in arithmetic combinatorics concerns the existence of certain differences in dense sets of integers. For $X\subseteq \Z$ and $N\in \N$, let $$D(X,N)=\max \left\{|A|: A\subseteq [N], \ (A-A)\cap X \subseteq \{0\} \right\},$$ where $[N]=\{1,2,\dots,N\}$ and $A-A=\{a-b: a,b\in A\}$. In other words, $D(X,N)$ is the threshold such that a subset of $[N]$ with more than $D(X,N)$ elements necessarily contains two distinct elements that differ by an element of $X$. Related definitions in an infinite setting are as follows.

\begin{definition} For $A\subseteq \N$, we define the \textit{density} of $A$ by $\delta(A)=\lim_{N\to \infty} |A\cap[N]|/N$, provided this limit exists, and we define the \textit{upper density}, $\bar{\delta}(A)$, by replacing the limit with limsup. For $X\subseteq \Z$, we define \begin{equation*}\mu(X) = \sup \left\{\bar{\delta}(A): A\subseteq \N, \ (A-A)\cap X\subseteq \{0\}\right\}. \end{equation*} We refer to $A\subseteq \Z$ satisfying $(A-A)\cap X \subseteq \{0\}$ as an \textit{$X$-set}, and we say that $X$ is \textit{intersective} if $\mu(X)=0$. This latter terminology is due to Ruzsa \cite{Ruz3}.
\end{definition}

Cantor and Gordon \cite{CG}, Haralambis \cite{NMH}, and Gupta \cite{Gupta} have investigated $\mu(X)$ in some special cases where $X$ is finite, and we borrow some notation and terminology developed in those papers. However, the majority of existing literature concerns infinite $X$, specifically the determination of whether $X$ is intersective, and quantitative estimates for $D(X,N)$ in intersective cases. The most well-studied cases are the squares, i.e. $X=S=\{n^2: n\in \N\}$, and $X=\mathcal{P}-1=\{p-1: p\in \P\}$, where $\P$ denotes the set of primes, originating from questions of Lov\'asz and Erd\H{o}s, respectively. The intersectivity of the squares was established independently by Furstenberg \cite{Furst} and S\'ark\"ozy \cite{Sark1}, the latter of whom also established intersectivity of $\P-1$ \cite{Sark3}. These results have spawned a great deal of refinements and extensions, and in Section \ref{kbs}, we provide, via tables and accompanying context, a survey of known upper and lower bounds on $D(X,N)$ in both classical and alternative intersective cases.

In addition to this summary of known results for intersective sets, we also explore $D(X,N)$ in some special nonintersective cases. In particular, we establish the following exact formulas in Section \ref{forms}, which are compatible with a more general conjecture we state in Section \ref{prelim}. Throughout, we use $1_E$ to denote the characteristic function of a set $E$.

\begin{theorem}\label{pform} The formula $D(\P,N)=\lceil N/4 \rceil + 1_E(N)$ holds for all $N\in \N$, where $E=\{2,3,4,11,12\}$. 
\end{theorem}

\begin{theorem}\label{sform} The formula $D(S+1,N)=\lceil N/3 \rceil + (1_E+1_{\{9,24\}})(N)$ holds for all $N\in \N$, where $$E=\{2,3,5,6,8,9,10,11,12, 17,18,20,21, 23, 24,25,26, 27\}.$$
\end{theorem}

In Section \ref{compsec}, we shift to a computational approach to the well-studied quantities $D(S,N)$ and $D(\P-1,N)$. Specifically, we leverage existing algorithms for finding maximum cliques in graphs to determine exact values of $D(S,N)$ for $N \leq 300$ and $D(\P-1,N)$ for $N \leq 500$, which have not previously appeared in the literature. We provide pseudocode descriptions of our computations, and also highlight some examples of notably dense sets lacking square or shifted prime differences.

\section{Preliminaries} \label{prelim}
Before proceeding further, we establish the connection between the finitary and infinitary formulations given in Section \ref{intro}. The following is similar to \cite[Theorem 1]{Ruz3}, but with a different approach and construction. 

\begin{theorem} For every $X\subseteq\Z$, we have $D(X,N)/N \to \mu(X)$ as $N\to \infty$. Futher, there exists an $X$-set $A\subseteq \N$ with $\bar{\delta}(A)=\mu(X)$. 
\end{theorem}

\begin{proof} Fix $X\subseteq \Z$. To first establish an upper bound on $\mu(X)$, suppose $A\subseteq \N$ is an $X$-set. By definition of $D(X,N)$, $A$ satisfies $|A\cap [N]|\leq D(X,N)$ for all $N\in \N$. Dividing by $N$ and passing to limits, we have \begin{equation}\label{ub} \bar{\delta}(A) = \limsup_{N\to \infty} \frac{|A\cap [N]|}{N}\leq \liminf_{N\to \infty} \frac{D(X,N)}{N}. \end{equation} For the lower bound, let $\delta= \limsup_{N\to \infty} D(X,N)/N$, so it suffices to construct an $X$-set of upper density at least $\delta$. For every $\epsilon>0$, we know by definition of $\limsup$ that for every $M>0$, there exists $N\geq M$ and an $X$-set $A\subseteq [N]$ with $|A|>  (\delta-\varepsilon)N$. We claim something stronger, that for every $N\in \N$ there exists an $X$-set $A\subseteq[N]$ with $|A|\geq \delta N$.

\noindent To verify the claim, first fix $\epsilon>0$, $N\in \N$, and an $X$-set $A' \subseteq [N']$ with $N'>4N/\epsilon$ and $|A'|>(\delta-\epsilon/4)N'$. Let $m=\lfloor N'/N \rfloor$, noting that $$|A'\cap[mN]|\geq |A'|-N > (\delta-\epsilon/4)N'-N > (\delta-\epsilon/2)N'\geq (\delta-\epsilon/2)mN.$$ Partitioning $[mN]$ into $m$ disjoint intervals of length $N$, the pigeonhole principle guarantees that $A'$ has density greater than $\delta-\epsilon$ on at least one interval. Since differences are invariant under translation, this high density interval corresponds to an $X$-set $A\subseteq [N]$ with $|A|>(\delta-\epsilon)N$. However, since $|A|$ is an integer and $\epsilon$ was chosen independent of $N$, we can choose $\epsilon$ small enough that $|A|\geq \lceil (\delta-\epsilon)N \rceil \geq \delta N$.

\noindent We now construct an infinite $X$-set as the union of an increasing sequence of finite sets. For each $N\in \N$, we let $G(N)=\{B\subseteq [N]: B \text{ is an }X\text{-set}, \ |B|\geq \delta N \}$, which the claim verified above ensures is nonempty. By the pigeonhole principle, for every $k\in \N$ and every $B\in G(2^k)$, there exist $B_j\in G(2^j)$ for $0\leq j \leq k$, where $B_k=B$ and a translation of $B_j$ is either the first or second half of $B_{j+1}$ for $0\leq j \leq k-1$. In this setting, if $i<j$, we say $B_i$ \textit{extends to} $[2^j]$. 

\noindent We start the process with $A_0=\{1\}$, which certainly extends to $[2^k]$ for all $k\in \N$. Inductively, suppose $j\geq 0$ and $A_{j}$ is a translation containing $1$ of an element of $G(2^j)$, which extends to $[2^k]$ for all $k>j$. We consider the infinitely many extensions of $A_{j}$, noting that, each time, a translation of $A_{j}$ is half of a translation of a set $B\in G(2^{j+1})$. Since $G(2^{j+1})$ is finite, there must exist $\tilde{A}_{j+1}\in G(2^{j+1})$ which occurs as this set $B$ infinitely often, so in other words $\tilde{A}_{j+1}$ also extends to $[2^k]$ for all $k>j+1$. We then define $A_{j+1}$ as the translation of $\tilde{A}_{j+1}$ that makes $A_j \subseteq A_{j+1}$. In other words, $A_{j+1}$ is obtained by adding a well-chosen block of length $2^j$ to either the left or right of $A_j$, maintaining density of at least $\delta$ and maintaining the property of extending to all higher levels. In particular, for each $j\geq 0$, $A_j$ is contained in an interval of length $2^j$ containing 1, which may contain both positive and negative integers.  

\noindent Finally, we define $A=\cup_{j=0}^{\infty} A_j$, so $A\subseteq \Z$ is an $X$-set, and for every $j\in \N$, $A$ has density at least $\delta$ on an interval of length $2^j$ containing $1$. This ensures that the $\limsup$ of either $|A\cap[N]|/N$ or $|(-A)\cap[N]|/N$ is at least $\delta$. Defining $A'$ to be either $A\cap \N$ or $(-A) \cap \N$, we have $\mu(X)\geq \bar{\delta}(A') \geq \delta,$ which combines with \eqref{ub} to complete the proof.\end{proof}

The following is a tangible necessary condition for intersectivity. 

\begin{definition} We say $X\subseteq \Z$ is \textit{locally intersective} if $X$ contains a nonzero multiple of every positive integer. This is equivalent to the statement that, for every infinite $A\subseteq \N$ and every $m\in \N$, the congruence $x-y\equiv z \ (\text{mod }m)$ is solvable with distinct $x,y\in A$ and $z\in X$, which is the motivation for the terminology. 
\end{definition}

\noindent Clearly intersectivity implies local intersectivity, since if $X$ contains no nonzero multiples of $m\in \N$, then the set $A=\{n\in \N: n\equiv 1 \ (\text{mod }m)\}$ is an $X$-set, hence $D(X,N)\geq \lceil N/m \rceil$ and $\mu(X)\geq 1/m$. The converse, however, is false. For one large family of counterexamples, recall that a \textit{lacunary sequence} $\{x_n\}$ of positive numbers satisfies $x_{n+1}\geq rx_n$ for some fixed $r>1$. It is known that a finite union of lacunary sequences in $\N$ is \textit{not} intersective (see \cite{peres}, for example), but many lacunary sequences, such as the Fibonacci sequence or the sequence of factorials, are locally intersective. Further, the failure of the converse is not solely based on the sparseness of lacunary sequences. Another family of counterexamples is as follows: fix an irrational number $\alpha>5$, let $A=\{\lfloor n\alpha \rfloor:n\in \N\}$, and let $X=\{n\in \N: |n-(k+1/2)\alpha|<1 \text{ for some }k\in \N\}$. By equidistribution, $X$ is locally intersective with $\delta(X)=2/\alpha$, while $A$ is an $X$-set satisfying $\delta(A)=1/\alpha$, hence $X$ is not intersective. 

In contrast with the counterexamples provided above, it follows from a theorem of Kamae and Mend\`es France \cite{KMF} that if $X=h(\Z)$ for a polynomial $h\in \Z[x]$, then $X$ is intersective if and only if it is locally intersective. For this reason, polynomials with locally intersective image, in other words nonzero polynomials with a root at every modulus, are called \textit{intersective polynomials}. Similarly, the set $\P-1$ (resp. $\P+1$) was identified as a candidate for intersectivity because of its local intersectivity, as for every $m\in \N$ there are plenty of primes congruent to $1$ (resp. $-1$) modulo $m$, while all other shifts of $\P$ fail to be locally intersective. More generally, for $h\in \Z[x]$, $h(\P)$ being intersective is equivalent to $h$ having a root modulo $m$ that is coprime to $m$ for every $m \in \N$, a condition known as \textit{$\P$-intersectivity} (see \cite{Rice}).
  
The underlying principle in the previous paragraph can be summarized as follows: for polynomial images of both $\Z$ and $\P$, local obstructions to intersectivity are the \textit{only} obstructions. This idea can be naturally extended to nonintersective cases: the maximal density of an $X$-set, if $X$ is built from polynomials and primes, should be determined by the maximal `local avoidance' of $X$-differences. We summarize this philosophy with the following general conjecture. 

\begin{conjecture} \label{locconj} For a nonzero polynomial $h\in \Z[x]$ and $X=h(\Z)$ or $X=h(\P)$, we have $$\mu(X)=\sup_{m\in \N} \frac{d_X(m)}{m},$$ where $d_X(m)=\max\{|A|: A\subseteq \Z/m\Z, \ (A-A)\cap X_m=\emptyset\},$ and $X_m$ is the set of congruence class modulo $m$ that intersect $X$.
\end{conjecture}
\noindent We note that $X$ is locally intersective if and only if $d_X(m)=0$ for all $m\in \N$. The purported lower bound for $\mu(X)$ can be quickly established with $X$-sets formed from unions of congruence classes, so the content of Conjecture \ref{locconj} lies in the upper bound.

\section{Primes and polynomials: exact formulas and estimates} \label{forms}

In this section we provide exact formulas or estimates for $D(X,N)$ in some special, nonintersective cases, which are compatible with Conjecture \ref{locconj}, including Theorems \ref{pform} and \ref{sform}. We frequently rely on the translation invariance of differences, meaning $x-y=(x+t)-(y+t)$, which in particular makes $D(X,N)$ a subadditive function of $N$. We include a proof of this standard fact below for completeness.

\begin{lemma}\label{tinv} If $N,M\in \N$ and $X\subseteq \Z$, then $D(X,N+M)\leq D(X,N)+D(X,M)$.
\end{lemma}

\begin{proof} Suppose $N,M\in \N$ and $X\subseteq \Z$. Fix an $X$-set $A\subseteq [N+M]$ with $|A|=D(X,N+M)$. Letting $B=A\cap[N]$ and $C\cap[N+1,N+M]$, we see that both $B$ and $C-N=\{c-N: c\in C\}\subseteq [M]$ are $X$-sets, the latter by translation invariance. By definition of $D$, we have $|B|\leq D(X,N)$ and $|C|=|C-N|\leq D(X,M)$, so $D(X,N+M)=|A|=|B|+|C|\leq D(X,N)+D(X,M)$, as required.
\end{proof}

\noindent We now establish Theorems \ref{pform} and Theorem \ref{sform}, the former by hand and the latter with computer assistance.

\begin{proof}[Proof of Theorem \ref{pform}] The following sets establish the lower bound for  $N\in E$: $\{1,2\}$ for $N=2,3,4$, and $\{1,2,10,11\}$ for $N=11,12$. The lower bound $D(\P,N)\geq \lceil N/4 \rceil$ for all $N\notin E$ is exhibited by $\{n\leq N: n\equiv 1 \pmod{4}\}$. We next show that $D(\P,8)=2$, and hence also $D(\P,N)=2$ for $3\leq N \leq 7$, and $D(\P,N)=4$ for $N=11\leq N \leq 16$. Suppose $A\subseteq [8]$ has no prime differences. By translation invariance, we can assume $1 \in A$ and, consequently, $3,4,6,8 \notin A$. If any $2,5$ or $7$ is present in $A$, then neither of the other two can be, as the difference between each of them is prime. Therefore, $|A|\leq 2$. Now suppose $A\subseteq [10]$ has no prime differences. By translation invariance, we can assume $1 \in A$. If $2$, $5$, or $7$ lie in $A$, then at most one of $9$ and $10$ can also be present in $A$, as each of $2,5,$ and $7$ differ by a prime from at least one of $9$ and $10$. Therefore, $|A| \leq 3$ and $D(\mathcal{P},10)=D(\P,9)=3.$ 

\noindent Suppose $A \subseteq [20]$ has no prime differences, and assume $1\in A$.  The interval $[20]$ can be split into two sets of eight ($[8], [9,16]$) and one set of four ($[17,20]$). As $D(\mathcal{P},8)=2$, we know $|A|\leq 2 + 2 +|A\cap[17,20]|$. Since $1 \in A$, we know  $18,20 \notin A$. This leaves only $17$ and $19$ as possible elements of $|A\cap[17,20]|$, and since they are separated by a prime difference of 2, only one of these numbers can be in $A$, hence $|A|\leq 5$ and $D(\P,20)=5$, and also $D(\P,N)=5$ for $N=17,18,19$. Finally, for $N>20$, $[N]$ can be partitioned into $[8],[9,16],\dots,[N-k-7,N-k], [N-k+1,N]$, where $k\in \{0,1,10,19,20,5,6,7\}$ with $k\equiv N\pmod{8}$, and Lemma \ref{tinv} is applied to establish $D(\P,N)=\lceil N/4 \rceil$.
\end{proof}

To clarify the compatibility of Theorem \ref{pform} with Conjecture \ref{locconj}, we note that $d_{\P}(4)=1$ because there are no primes divisible by $4$, so $1/4 = \mu(\P)\geq \sup_{m\in \N} d_{\P}(m)/m\geq d_{\P}(4)/4=1/4,$ and hence all quantities involved are equal to $1/4$. Similarly for Theorem \ref{sform}, $h(x)=x^2+3$ has no root modulo $3$, so $d_{S+3}(3)=1$ and $\mu(S+3)\geq \sup_{m\in \N} d_{S+3}(m)/m\geq 1/3$, and the formula proven below establishes equality.

\begin{proof}[Proof of Theorem \ref{sform}] Maximum clique computations, as described in Section \ref{compsec}, establish the formula for $N\leq 50$, the most exceptional cases being $D(S+1,9)=5$ and $D(S+1,24)=10$, with lower bounds exhibited by the sets $\{1,2,5,8,9\}$ and $\{1,2,5,8,9,16,17,20,23,24\}$, respectively. The lower bound $D(S+1,N)\geq \lceil N/3 \rceil$ for all $N$ is exhibited by $\{n\leq N: n\equiv 1 \ (\text{mod }3)\}$.  Then, for $N>42$, $[N]$ can be partitioned into $[15],[16,30],\dots,[N-k-14,N-k], [N-k+1,N]$, where $k\in \{0,1,32,33,4,35,36,7,8,39,40,41,42,13,14\}$ with $k\equiv N\pmod{15}$, and Lemma \ref{tinv} is applied to establish $D(S+1,N)\leq \lceil N/3 \rceil$. \end{proof}

In the two preceeding results, the formula $D(X,N)=\lceil N/m^*\rceil$ holds for sufficiently large $N$, where $m^*=\min\left\{m\in \N: (X\cap m\Z) =\emptyset\right\}$. We note that this does not hold in general for $X=h(\Z)$, $h\in \Z[x]$, as the supremum in Conjecture \ref{locconj} is not necessarily $1/m^{*}$. For example, the set $X=S+3$ intersects $2\Z$, $3\Z$, and $4\Z$, but $\mu(S+3)\geq 1/4$ as exhibited by the set $A=\{n\in \N: n\equiv 1\text{ or }3  \ (\text{mod }8) \}$, in other words $d_{S+3}(8)=2$. Further, even when numerical evidence suggests $\mu(X)=1/m^*$, we can still have $D(X,N)>\lceil N/m^*\rceil$ infinitely often, as shown in the following example.    
 
\begin{theorem}\label{Sp2} The lower bound $D(S+2,N)>\lceil N/4 \rceil$ holds unless $N = 4k^2+5$ for some $k\in \N$. More specifically, letting $j=\lceil N/4 \rceil$, we have $$D(S+2,N)\geq j + \begin{cases} 1 & N\equiv 2 \ (\textnormal{mod }4) \text{ or }(N\equiv 0 \text{ or }3 \ (\textnormal{mod }4) \text{ and }4j-2\in S+2 ) \\ \ & \text{or }(N\equiv 1 \ (\textnormal{mod }4) \text{ and }4j-6 \notin S+2) \\ 2 & N\equiv 0 \text{ or }3 \ (\textnormal{mod }4) \text{ and }4j-2\notin S+2 \\ 0 & N\equiv 1 \ (\textnormal{mod }4) \text{ and }4j-6 \in S+2  \end{cases}. $$

\end{theorem} 
\begin{proof} Let $X=S+2$ and note that elements of $X$ are positive and congruent to $2$ or $3$ modulo $4$. Fix $N\in \N$, let $j=\lceil N/4 \rceil$, and let $A=\{1,2,6,10,\dots, 4j-6\}$, so $|A|=j$.  Further, differences in $A$ are either $0$ modulo $4$, positive and $1$ modulo $4$, or negative and $3$ modulo $4$, hence none lie in $X$, so $A$ is an $X$-set. If $N\equiv 2 \pmod{4}$, then $4j-2 =N$, and $A\cup\{4j-2\}$ is an $X$-set for the same reasons as $A$, completing the proof in this case. If $N \equiv 1 \pmod{4}$, then $A'=A\cup\{4j-5\} \subseteq [N-2]$, and differences in $A'$ are either differences in $A$, positive and $1$ modulo $4$, negative and $3$ modulo $4$, or $4j-6$. The only one of these that can possibly lie in $X$ is $4j-6$, so $A'$ is an $X$-set unless $4j-6\in X$, completing the proof for $N \equiv 1 \pmod{4}$. Moreover, if $4j-6=n^2+2$, $n$ must be even, so $4j-6=(2k)^2+2$, hence $N=4j-3=4k^2+5$. 

\noindent Finally, if $N\equiv 0\textnormal{ or }3 \ (\textnormal{mod }4)$, we apply the previous case for $N+1$ or $N+2$, respectively, which are congruent to $1$ mod $4$, noting that $\lceil(N+1)/4\rceil$ and $ \lceil(N+2)/4\rceil$ are equal to $j+1$ in the respective cases.
\end{proof}

\noindent While we do not provide a proof that $\mu(S+2)=1/4$, maximum clique computations, as discussed in the next section, yield that the lower bound in Theorem \ref{Sp2} holds with equality for all $51\leq N\leq 150$. In particular, $\mu(S+2)\leq D(S+2,149)/149=38/149\approx 0.255$.






\section{Maximum cliques and computational data} \label{compsec}

In an effort to collect computational data for the quantities $D(S,N)$ and $D(\P-1,N)$, we make a connection between forbidden differences and graph theory, as was done previously in \cite{Lewko}.

\begin{definition} For a graph $G$, a \textit{clique} of $G$ is a complete subgraph $C\subseteq G$. Further, $C$ is called a \textit{maximum clique} of $G$ if $|C|\geq |C'|$ for all cliques $C'\subseteq G$.
\end{definition}

\noindent Note the distinction between a maximum clique and a \textit{maximal clique}, which is a clique that is not contained in any strictly larger clique. The problem of developing algorithms that take a graph as input and produce a single maximum clique, a list of all maximum cliques, or just the size of a maximum clique, is a well-studied problem in computer science that is known to be NP-complete. Maximum cliques are directly connected with the forbidden differences problems discussed in this paper in the following way: given $X\subseteq \Z$ and $N\in \N$, we can build a graph $G$ with vertices in $[N]$ by connecting vertex $i$ with vertex $j$ if and only if neither $i-j$ nor $j-i$ lie in $X$. Then, $D(X,N)$ is precisely the size of a maximum clique in $G$. Further, by translation invariance, we can instead identify the vertices of the graph with $\{0,1,\dots,N-1\}$, and restrict our search to maximum cliques containing $0$. In particular, instead of working with the full graph $G$, we can restrict to  the neighbors of $0$, and add $0$ to the returned maximum clique. A pseudocode description of the construction of this `ZeroNeighborhoodGraph' is given below.

\begin{algorithm}
\textbf{ZeroNeighborhoodGraph}
\begin{algorithmic}
\State{\textbf{Input:} $N\in \N,X\subseteq \Z$}
\State $V \gets \{\}$
\State $E \gets \{\}$
\For{$1\leq i \leq N-1$}
	\If{$i \notin X$ and $-i\notin X$}  
	\State $V\gets V\cup\{i\}$
	\For{$1\leq j<i$}
		\If{$j\in V$ and $j-i \notin X$ and $i-j\notin X$} $E\gets E\cup\{\{i,j\}\}$
		\EndIf
	\EndFor
	\EndIf
	\State $i \gets i+1$
	
	\EndFor \\
\Return $(V,E)$
\end{algorithmic}
\end{algorithm}

We use the existing MaxCliqueDyn algorithm \cite{mcd}, developed in \cite{konc}, implemented in C++, which takes in a graph $G=(V,E)$ as input and outputs $(M,C)$, where $C\subseteq G$ is a maximum clique and $M=|C|$. Rather than run the algorithm for each $N$, which would result in a great deal of redundancy, we instead take a `top-down' approach. For example, suppose we input $N=280$ and $X=S$, build the ZeroNeighborhood graph $G$, compute $(53,C)=\text{MaxCliqueDyn}(G)$, and add $0$ to $C$ to get a maximum clique $C'$ of the full graph $G'$, with $|C'|=54$. Suppose further that, with elements sorted in increasing order, $C'=\{0,\dots,268\}$. In this case, we know that $C'$ is a maximum clique not only for $N=280$, but in fact for all $269\leq N \leq 280$, meaning $D(S,N)=54$ for $269\leq N\ \leq 280$, and we can skip over $11$ values of $N$ and continue the top-down progression by inputting $N=268$. A pseudocode description of this `cascade' process, for $N$ in a chosen range, is provided below.

\begin{algorithm}[H]
\textbf{MaxCliqueCascade}
\begin{algorithmic}
\State{\textbf{Input:} Min, Max$\in \N$, $X\subseteq\Z$}
\State $N \gets$ Max
\While{$N>$ Min} 
	\State Print $N$
	\State $G \gets$ ZeroNeighborhoodGraph$(N,X)$
	\State $(M,C) \gets$ MaxCliqueDyn$(G)$
	\State Print $M$ and $C$ 
	
	\State $N' \gets \max(C)$
	\Comment $D(X,n)=M+1$ with optimal example $A=\{1\}\cup (C+1)$ for all $N'+1\leq n \leq N$
	\State $N \gets N'$
	
	\EndWhile
\end{algorithmic}
\end{algorithm}
\noindent With computing time totaling a few days on personal laptops, we collected the following values for $D(S,N)$.

\begin{table}[H]
\caption{Exact values for $D(S,N)$ computed with MaxCliqueCascade.}
\centering
\begin{minipage}[c]{.3\linewidth}
\centering
\renewcommand{\arraystretch}{1.5}
\begin{tabular}{|| c || c ||}
\hline
$N$ & \textbf{$D(S,N)$} \\
\hline
$1 \leq N \leq 2$ & $1$ \\
\hline
$3 \leq N \leq 5$ & $2$ \\
\hline
$6 \leq N \leq 7$ & $3$ \\
\hline
$8 \leq N \leq 10$ & $4$ \\
\hline
$11 \leq N \leq 12$ & $5$ \\
\hline
$13 \leq N \leq 15$ & $6$ \\
\hline
$16 \leq N \leq 17$ & $7$ \\
\hline
$18 \leq N \leq 20$ & $8$ \\
\hline
$21 \leq N \leq 22$ & $9$ \\
\hline
$23 \leq N \leq 34$ & $10$ \\
\hline
$35 \leq N \leq 37$ & $11$ \\
\hline
$38 \leq N \leq 42$ & $12$ \\
\hline
$43 \leq N \leq 47$ & $13$ \\
\hline
$48 \leq N \leq 52$ & $14$ \\
\hline
$53 \leq N \leq 57$ & $15$ \\
\hline
$58 \leq N \leq 65$ & $16$ \\
\hline
$66 \leq N \leq 67$ & $17$ \\
\hline
$68 \leq N \leq 70$ & $18$ \\
\hline
$71 \leq N \leq 72$ & $19$ \\
\hline
\end{tabular}
\end{minipage}
\begin{minipage}[c]{.3\linewidth}
\centering
\renewcommand{\arraystretch}{1.5}
\begin{tabular}{|| c || c ||}
\hline
$N$ & \textbf{$D(S,N)$} \\
\hline
$73 \leq N \leq 80$ & $20$ \\
\hline
$81 \leq N \leq 85$ & $21$ \\
\hline
$86 \leq N \leq 91$ & $22$ \\
\hline
$92 \leq N \leq 96$ & $23$ \\
\hline
$97 \leq N \leq 101$ & $24$ \\
\hline
$102 \leq N \leq 106$ & $25$ \\
\hline
$107 \leq N \leq 111$ & $26$ \\
\hline
$112 \leq N \leq 117$ & $27$ \\
\hline
$118 \leq N \leq 119$ & $28$ \\
\hline
$120 \leq N \leq 124$ & $29$ \\
\hline
$125 \leq N \leq 130$ & $30$ \\
\hline
$131 \leq N \leq 132$ & $31$ \\
\hline
$133 \leq N \leq 137$ & $32$ \\
\hline
$138 \leq N \leq 143$ & $33$ \\
\hline
$144 \leq N \leq 145$ & $34$ \\
\hline
$146 \leq N \leq 150$ & $35$ \\
\hline
$151 \leq N \leq 156$ & $36$ \\
\hline
$157 \leq N \leq 158$ & $37$ \\
\hline
$159 \leq N \leq 163$ & $38$ \\
\hline
\end{tabular}
\end{minipage}
\begin{minipage}[c]{.3\linewidth}
\centering
\renewcommand{\arraystretch}{1.5}
\begin{tabular}{|| c || c ||}
\hline
$N$ & \textbf{$D(S,N)$} \\
\hline
$164 \leq N \leq 188$ & $39$ \\
\hline
$189 \leq N \leq 198$ & $40$ \\
\hline
$199 \leq N \leq 202$ & $41$ \\
\hline
$203 \leq N \leq 205$ & $42$ \\
\hline
$206 \leq N \leq 207$ & $43$ \\
\hline
$208 \leq N \leq 218$ & $44$ \\
\hline
$219 \leq N \leq 222$ & $45$ \\
\hline
$223 \leq N \leq 235$ & $46$ \\
\hline
$236 \leq N \leq 241$ & $47$ \\
\hline
$242 \leq N \leq 247$ & $48$ \\
\hline
$248 \leq N \leq 252$ & $49$ \\
\hline
$253 \leq N \leq 257$ & $50$ \\
\hline
$258 \leq N \leq 262$ & $51$ \\
\hline
$263 \leq N \leq 265$ & $52$ \\
\hline
$266 \leq N \leq 268$ & $53$ \\
\hline
$269 \leq N \leq 282$ & $54$ \\
\hline
$ 283\leq N \leq 284$ & $55$ \\
\hline
$ 285 \leq N \leq 287$ & $56$ \\
\hline
$ 288 \leq N \leq 292$ & $57$ \\
\hline
$ 293 \leq N \leq 300$ & $58$ \\
\hline
\end{tabular}
\end{minipage}
\end{table}
\newpage
\noindent While we have examples of optimal square-difference free sets for all $N\leq 300$, we omit them here for the sake of brevity. However, we provide one notable example below, foreshadowed prior to the pseudocode description of MaxCliqueCascade: a square difference free set $A\subseteq [269]$ with $|A|=54$, which remains optimal until $N$ reaches $283$.

\begin{align*} A = \{& 1, 4, 6, 9, 11, 14, 16, 21, 28, 33, 38, 48,  51, 59,
66, 72, 79, 86, 89, 94,  96, 107, \\& 113, 118, 124, 126,
131, 139, 144, 146, 152, 157, 163, 174, 176, 181,
184, 191, \\& 204,  211, 214, 219, 222, 232, 237, 242,
249, 254, 256, 259, 261, 264, 266, 269\}\end{align*}

\noindent In addition, with computing time totaling about one day on a personal laptop, we collected the following values for $D(\P-1,N)$. The algorithm runs faster for $X=\P-1$ because the corresponding graphs are substantially sparser.

\begin{table}[H]
\caption{Exact values for $D(\P-1,N)$ computed with MaxCliqueCascade.}
\centering
\begin{minipage}[c]{.48\linewidth}
\centering
\renewcommand{\arraystretch}{1.5}
\begin{tabular}{|| c || c ||}
\hline
$N$ & \textbf{$D(\mathcal{P}-1,N)$} \\
\hline

$1\leq N \leq 3$ & $1$ \\
\hline
$4\leq N \leq 8$ & $2$ \\
\hline
$9 \leq N \leq 11$ & $3$ \\
\hline
 $12 \leq N \leq 32$ & $4$ \\
\hline
$33 \leq N \leq 35$ & $5$ \\
\hline
$36 \leq N \leq 48$ & $6$ \\
\hline
$49 \leq N \leq 51$ & $7$ \\
\hline
$52 \leq N \leq 64$ & $8$ \\
\hline
$65 \leq N \leq 67$ & $9$ \\
\hline
$68 \leq N \leq 104$ & $10$ \\
\hline
$105 \leq N \leq 107$ & $11$ \\
\hline
108 $\leq N \leq$ 132 & 12 \\
\hline
133 $\leq N \leq$ 135 & 13 \\
\hline

\hline
\end{tabular}
\end{minipage}
\begin{minipage}[c]{.48\linewidth}
\centering
\renewcommand{\arraystretch}{1.5}
\begin{tabular}{|| c || c ||}
\hline
$N$ & \textbf{$D(\mathcal{P}-1,N)$} \\
\hline

$136 \leq N \leq 152$ & $14$ \\
\hline
$153 \leq N \leq 155$ & $15$ \\
\hline
$156 \leq N \leq 208$ & $16$ \\
\hline
$209 \leq N \leq 211$ & $17$ \\
\hline
$212 \leq N \leq 216$ & $18$ \\
\hline
$217 \leq N \leq 219$ & $19$ \\
\hline
$220 \leq N \leq 242$ & $20$ \\
\hline
$243 \leq N \leq 245$ & $21$ \\
\hline
$246 \leq N \leq 298$ & $22$ \\
\hline
$299 \leq N \leq 301$ & $23$ \\
\hline
$302 \leq N \leq 488$ & $24$ \\
\hline
$489 \leq N \leq 491$ & $25$ \\
\hline
$492 \leq N \leq 500$ & $26$ \\
\hline

\hline
\end{tabular}
\end{minipage}
\end{table}
As with the squares, we include the most notably dense example of a $(p-1)$-difference free set, in this case $A\subseteq [302]$ with $|A|=24$, which is optimal until $N$ reaches $489$: $A = \{ 1, 4, 57, 60, 65, 68, 91, 94, 141, 144, 155,\\ 158, 175, 178,
189, 192, 209, 212, 265, 268, 273, 276, 299, 302\}$
  
\section{Survey of Known Bounds} \label{kbs}

The following is a (to our knowledge) comprehensive survey of known upper and lower bounds on $D(X,N)$ for intersective $X$, with results ranging from the mid 1970s to 2025. We separate the bounds into three tables: squares, $\P-1$, and others, the last of which predominantly consists of more general polynomial images. An asterisk ($^*$) on a citation indicates the current best-known upper bound of its type, while a double asterisk ($^{**}$) indicates the current best-known lower bound of its type. We use the Vinogradov symbol $\ll$ to denote `less than a constant times', and implied constants may depend on parameters \textit{other than} $N$. We use $c$ to denote a positive constant, which must also be independent of $N$.

\setlength{\extrarowheight}{3pt}

\begin{table}[H]

\centering

\caption{Bounds on $D(S,N)$, where $S$ is the set of squares} \label{Stable}


\begin{tabular}{||c||c||c||c||} 

\hline

Bound on $D(S,N)$ & Due to \\

\hline\hline

$o(N)$  & Furstenberg  \cite{Furst} \\

\hline

$\leq N(\log N)^{-\frac{1}{3}+o(1)}$ & S\'ark\"ozy \cite{Sark1}  \\

\hline

$\gg N^{\frac{1+\frac{\log 7}{\log 65}}{2}} = N^{0.7331\dots}$ & Ruzsa \cite{Ruz2}  \\

\hline

$\ll N(\log N)^{-\frac{1}{12}\log\log\log\log N}$ & Pintz, Steiger, Szemer\'edi \cite{PSS} \\

\hline

$\gg N^{\frac{1+\frac{\log 12}{\log 205}}{2}} = N^{0.7334\dots}$ & Lewko \cite{Lewko}$^{**}$   \\

\hline

$\ll N(\log N)^{-c\log\log\log N}$ & Bloom, Maynard \cite{BloomMaynard} \\

\hline

$\ll N\exp(-c\sqrt{\log N})$ & Green, Sawhney \cite{GreenNew2}$^*$ \\

\hline

\end{tabular}

\end{table}
\begin{table}[H]

\centering

\caption{Bounds on $D(\P-1,N)$, where $\P-1 = \{p-1: p \text{ prime}\} = \{1,2,4,6,10,12,\dots\}$}


\begin{tabular}{||c||c||}

\hline

Bound on $D(\P-1,N)$ & Due to \\

\hline\hline

$\leq N(\log \log N)^{-2+o(1)}$ & S\'ark\"ozy \cite{Sark3}  \\

\hline

$\geq N^{\left(\frac{\log 2}{2}-o(1) \right)/\log\log N}$ & Ruzsa \cite{Ruz3}$^{**}$  \\

\hline

$\ll N(\log\log N)^{-c\log\log\log\log\log N}$ & Lucier \cite{Lucier2}   \\

\hline

$\ll N\exp(-c(\log N)^{1/4})$ & Ruzsa, Sanders \cite{Ruz}  \\

\hline

$\ll N\exp(-c(\log N)^{1/3})$ & Wang \cite{Wang}  \\

\hline

$\ll N^{1-c}$ & Green \cite{GreenNew}$^*$  \\

\hline

\end{tabular}

\end{table}

All of the best upper bounds in Tables \ref{Stable}-\ref{otherX} are obtained through Fourier analytic density increment arguments, with the exception of Green's \cite{GreenNew} breakthrough bound $D(\P-1,N)\ll N^{1-c}$, which follows from a quantitative improvement of the \textit{van der Corput property} for shifted primes. All of the lower bounds in Tables \ref{Stable}-\ref{otherX}, with the exception of the greedy algorithm, have their roots in constructions of Ruzsa \cite{Ruz2, Ruz3}, which bridge the modular version of the problem (i.e. forbidden differences in $\Z/m\Z$) and the integer version. In Table \ref{otherX}, we use $d_X(m)$ as defined in Conjecture \ref{locconj}. Further, the terms \textit{strongly Deligne} and \textit{$\P$-Deligne} refer to large families of intersective and $\P$-intersective multivariate polynomials, respectively, satisfying appropriate nonsingularity conditions, as defined in \cite{DR} and \cite{DR2}.

 \


\begin{table}[H] 

\centering
\small

\caption{Bounds on $D(X,N)$ for other $X$} 
\label{otherX}

\begin{tabular}{||c||c||c||} 
\hline
 
$X$ & Bound on $D(X,N)$ & Due to \\

\hline \hline \ & \ & \ \\

any $X\subseteq \N$ & $\geq \dfrac{N-1}{|X\cap[N]|+1}$ & greedy algorithm (see \cite{Lyall})   \\

\ & \ & \ \\

\hline

$h(\Z)$, $h\in \Z[x]$ intersective  & $o(N)$ & Kamae, Mend\'es-France \cite{KMF}   \\

\hline

$\{n^k : n\in \N\}$ & $\gg N^{\frac{k-1+\frac{\log d_X(m)}{\log m}}{k}},$ \ $m$ squarefree & Ruzsa \cite{Ruz2}$^{**}$ \\

\hline

$\{n^k : n\in \N\}$ & $\ll N(\log N)^{-\frac{1}{4}\log\log\log\log N}$ & Balog, Pelik\'an,   \\

\ & \ &  Pintz, Szemer\'edi \cite{BPPS}  \\

\hline

$h(\Z)$, $h\in \Z[x]$, & $\ll N/\log\log\log N$ & Slijep\v{c}evi\'c \cite{Slip}   \\

$h(0)=0$ & \ & \  \\

\hline

$h(\Z)$, $h\in \Z[x]$ & $\leq N(\log N)^{-\frac{1}{\deg(h)-1}+o(1)}$ & Lucier \cite{Lucier}   \\

intersective & \ & \  \\

\hline

$h(\P)$, $h\in \Z[x]$, & $\ll N/\log\log\log N$ & Li, Pan \cite{lipan}   \\

$h(1)=0$ & \ & \  \\
\hline

$h(\Z)$, $h\in \Z[x]$ & $\leq N(\log N)^{\left(-\frac{1}{\log 3}+o(1)\right)\log\log\log\log N}$& Hamel, Lyall, Rice \cite{HLR} \\

intersective, $\deg(h)=2$ & \ & \  \\

\hline

$h(\P)$, $h\in \Z[x]$, $\P$-intersective & $\leq N(\log N)^{-\frac{1}{2(\deg(h)-1)}+o(1)}$ & Rice \cite{Rice}   \\

\hline

$h(\Z)$, $h\in \Z[x]$ & $\leq N(\log N)^{\Big(-\frac{1}{\log\left(\frac{k^2+k}{2}\right)}+o(1)\Big)\log\log\log\log N}$ & Rice \cite{ricemax}   \\

intersective, $\deg(h)=k$  & \ & \   \\

\hline

$\{am^2+bmn+cn^2: m,n\in \N\}$ & $\ll N\exp(-c\sqrt{\log N})$ & Rice \cite{Ricebin}   \\

$a,b,c \in \Z$, $b^2-4ac\neq 0$  & \ & \   \\

\hline

$\{m^2+n^2 : m,n\in \N\}$ & $\gg \sqrt{N}$ & Younis \cite{Younis}$^{**}$   \\

\hline

$h(\Z^{\l})$,  & $\ll N\exp(-c(\log N)^{\mu}),$ & \   \\

$h\in \Z[x_1,\dots,x_{\l}],$ $\l\geq 2$,  & $\mu=\begin{cases} [(\deg(h)-1)^2+1]^{-1} & \text{if }\l=2 \\ 1/2 &\text{if }\l \geq 3  \end{cases} $ & Doyle, Rice \cite{DR}$^*$    \\

strongly Deligne  & \ & \  \\

\hline

$h(\Z)$, $h\in \Z[x]$ intersective & $\ll N(\log N)^{-c\log\log\log N}$ & Arala \cite{Arala}$^*$  \\

\hline

$h(\P)$, $h\in \Z[x]$, $\P$-intersective & $\ll N(\log N)^{-c\log\log\log N}$ & Doyle, Rice  \cite{DR2}$^*$  \\

\hline

$h(\P^{\l})$,  & $\ll N\exp(-c(\log N)^{\mu}),$ & \  \\

$h\in \Z[x_1,\dots,x_{\l}],$  & $\mu=\begin{cases} [2(\deg(h)-1)^2+6]^{-1} & \text{if }\l=2 \\ 1/4 &\text{if }\l \geq 3  \end{cases}$ & Doyle, Rice \cite{DR2}$^*$ \\

$\l \geq 2$, $\P$-Deligne  & \ & \  \\
\hline
\end{tabular}
\end{table}

\noindent \textbf{Acknowledgements:} This research was initiated during the Summer 2024 Kinnaird Institute Research Experience at Millsaps College. All authors were supported during the summer by the Kinnaird Endowment, gifted to the Millsaps College Department of Mathematics. At the time of completion, all authors except Alex Rice and Andrew Lott were Millsaps College undergraduate students. Alex Rice was partially supported by an AMS-Simons research grant for PUI faculty.  The authors would like to thank Drewrey Lupton and Philipp Birklbauer for their code-related assistance, and John Griesmer for his helpful insights and references.

\end{document}